\documentclass[12pt]{amsart}
\usepackage{amssymb,verbatim,enumerate,ifthen}
\usepackage[mathscr]{eucal}
\usepackage[utf8]{inputenc}
\usepackage[T1]{fontenc}
\usepackage{esint}
\newtheorem{thm}{Theorem}[section]
\newtheorem{cor}[thm]{Corollary}

\newtheorem{lem}[thm]{Lemma}

\theoremstyle{definition}

\newtheorem*{xrem}{Remark}

\numberwithin{equation}{section}
\def\eq#1{{\rm(\ref{#1})}}
\def\Eq#1#2{\ifthenelse{\equal{#1}{*}}
  {\begin{equation*}\begin{aligned}[]#2\end{aligned}\end{equation*}}
  {\begin{equation}\begin{aligned}[]\label{#1}#2\end{aligned}\end{equation}}}

\newcommand\dom{\mathrm{dom}}
\def\A{\mathscr{A}}

\def\M{\mathscr{M}}
\def\Nn{\mathscr{N}}
\def\P{\mathscr{P}}
\DeclareMathOperator{\Var}{Var}
\DeclareMathOperator{\Lip}{Lip}
\newcommand{\op}[1]{\mathop{\vphantom{\sum}\mathchoice
{\vcenter{\hbox{\LARGE $#1$}}}
{\vcenter{\hbox{\Large $#1$}}}{#1}{#1}}\displaylimits}
\newcommand{\bff}{\mathbf{f}}

\newcommand{\Sm}{\mathcal{S}}
\def\SmLip{\Sm^{Lip}}
\def\Mm{\op{\mathscr{M}}}
\newcommand{\Mik}[1]{P_{#1}}
\def\Ar{\op{\mathscr{A}}}

\newcommand\R{\mathbb{R}}
\newcommand\N{\mathbb{N}}
\newcommand\Z{\mathbb{Z}}
\newcommand\Mb{\mathbf{M}}
\newcommand\Q{\mathbb{Q}}

\newcommand{\QA}[1]{\A_{#1}}
\newcommand{\QAself}[1]{\textbf{A}_{#1}}
\newcommand{\oa}{\overline{a}}
\newcommand{\norm}[1]{\left\| #1 \right\| }
\newcommand{\abs}[1]{\left| #1 \right| }

\newcommand{\mgp}{\M_\otimes}

\newcommand{\dotvec}[3][SKIPPED]{
\ifthenelse{\equal{#1}{SKIPPED}}
  {#2,\dots,#3}
  {\underbrace{#2,\dots,#3}_{#1\text{ entries}}}
}

\newcommand{\Hc}[2][SKIPPED]{
\ifthenelse{\equal{#1}{SKIPPED}}
  {
    \ifthenelse{\equal{#2}{}}
      {\mathscr{H}}
      {\mathscr{H}(#2)}
  }
  {
    \ifthenelse{\equal{#2}{}}
      {\mathscr{H}_{#1}}
      {\mathscr{H}_{#1}(#2)}
  }
}
\newcommand{\Est}[2][SKIPPED]{
\ifthenelse{\equal{#1}{SKIPPED}}
  {
    \ifthenelse{\equal{#2}{}}
      {\mathscr{C}}
      {\mathscr{C}(#2)}
  }
  {
    \ifthenelse{\equal{#2}{}}
      {\mathscr{C}_{#1}}
      {\mathscr{C}_{#1}(#2)}
  }
}
\author[P. Pasteczka]{Pawe\l{} Pasteczka}
\address{Institute of Mathematics \\ Pedagogical University of Cracow \\ Podchor\k{a}\.zych str. 2, 30-084 Krak\'ow, Poland}
\email{pawel.pasteczka@up.krakow.pl}
\title{On the quasi-arithmetic Gauss-type iteration}
\begin{document}
\begin{abstract}
For a sequence of continuous, monotone functions $f_1,\dots,f_n \colon I \to \mathbb{R}$ ($I$ is an interval) we define the mapping 
$M \colon I^n \to I^n$ as a Cartesian product of quasi-arithmetic means generated by $f_j$-s. It is known that, for every initial vector, the iteration sequence of this mapping tends to the diagonal of $I^n$.

We will prove that whenever all $f_j$-s are $\mathcal{C}^2$ with nowhere vanishing first derivative, then this convergence is quadratic. Furthermore, the limit $\frac{\text{Var}\, M^{k+1}(v)}{(\text{Var}\, M^{k}(v))^2}$ will be calculated in a nondegenerated case.
\end{abstract}
\maketitle

\section{Introduction}
In 1800 (this year is due to \cite{ToaToa05}) Gauss introduced the arithmetic-geometric mean as a limit in the following two-term recursion:
\Eq{AGM:ITP}{
x_{k+1}=\frac{x_k+y_k}2, \qquad y_{k+1}=\sqrt{x_ky_k},
}
where $x_0=x$ and $y_0=y$ are two positive parameters. 
Gauss \cite[p.\ 370]{Gau18} proved that both $(x_k)_{k=1}^\infty$ and 
$(y_k)_{k=1}^\infty$ converge to a common limit, which is called arithmetic-geometric mean of the initial values $x_0$ 
and $y_0$. J.~M.~Borwein and P.~B.~Borwein \cite{BorBor87} extended some earlier ideas \cite{FosPhi84a,Leh71,Sch82} and 
generalized this iteration to a vector of continuous, strict means of an arbitrary length.
For several recent results about Gaussian product of means see the papers by Baj\'ak--P\'ales 
\cite{BajPal09b,BajPal09a,BajPal10,BajPal13}, by Dar\'oczy--P\'ales \cite{Dar05a,DarPal02c,DarPal03a}, 
by G{\l}azowska \cite{Gla11b,Gla11a}, by Matkowski \cite{Mat99b,Mat02b,Mat05,Mat13}, by Matkowski--P\'ales 
\cite{MatPal15}, and by the author \cite{Pas16a}.

Recall that for a given interval $I$, a \emph{mean} defined on $I$ is any function $\M \colon \bigcup_{n=1}^{\infty} I^n \to I$ such that $\min(a) \le \M(a) \le \max(a)$  for every admissible $a$. The mean is \emph{strict} if $\min(a)<\M(a)<\max(a)$ unless $a$ is a constant vector.

It is known, \cite[Theorem 8.2]{BorBor87}, that for every twice continuously differentiable, strict means $\M$, $\Nn$ and sequences
\Eq{*}{
x_{k+1}=\M(x_k,y_k), \qquad y_{k+1}=\Nn(x_k,y_k), \quad k \in \N_{+} \cup \{ 0 \}.
} 
the difference $|x_k-y_k|$ tends to zero quadratically for all $x_0=x$ and $y_0=y$.

Following \cite[section~8.7]{BorBor87}, we will consider the iteration of multidimensional means. Given a natural number $n \in \N$ and a vector of means $(\M_1,\dots,\M_n)$ defined on a common interval~$I$. Let us define the mapping $\Mb  \colon I^n \to I^n$ by 
\Eq{*}{
\Mb(a):=\big(\M_1(a),\dots,\M_n(a)\big), \qquad a \in I^n.
}

Whenever for every $i \in \{1,\dots,N\}$ the limit of its iteration sequence $\lim_{k \rightarrow \infty} [\Mb^k(a)]_i$ exist and do not depend on $i$, we call it  \emph{Gaussian product} of $(\M_i)$ and denote it by $\mgp(a)$. Matkowski used to call the mean $\mgp$ the \emph{$\Mb$-invariant mean}. Indeed, $\mgp$ can be characterized as a unique mean satisfying the equality $\mgp \circ \Mb = \mgp$ (cf. e.g. \cite{Mat99b}). He also proved that whenever all means are continuous and strict then $\mgp$ is a uniquely defined continuous and strict mean. 

Some special case is that for some $k_0 \in \N$ the vector $\Mb^{k_0}(a)$ is constant. Then, for all $k \ge k_0$, we have $\Mb^{k}(a)=\Mb^{k_0}(a)$. In particular each entry of this vector equals $\mgp(a)$. If it is the case for some nonconstant vector $a$, then we will call such iteration process to be \emph{degenerated}. It can be easily verified that under some mild condition regarding comparability of means iteration process is never degenerated. Such results are however outside the scope of this paper and are omitted.

Gauss' iteration process in a case when all means are quasi-arithmetic one will be of our interest. It was already under investigation in \cite{Pas16a}. We are going to continue the research in this area. In particular we will prove the multidimensional counterpart of \cite[Theorem 8.2]{BorBor87} in a case when all consider means are quasi-arithmetic. Furthermore we will show that, under some conditions, not only the convergence is quadratic, but also the characteristic ratio is closely related to so-called Arrow-Pratt index.

\section{Quasi-arithmetic means}
Quasi-arithmetic means were introduced in series of nearly simultaneous papers in a beginning of 1930s \cite{Def31,Kol30,Nag30} as a generalization of already mentioned family of power means. 
For a continuous and strictly monotone function $f \colon I \to \R$ ($I$ is an interval) and a vector $a=(a_1,a_2,\dots,a_n) \in I^n$, $n \in \N$ we define 
\Eq{*}{
\QA{f}(a):=f^{-1}\left( \frac{f(a_1)+f(a_2)+\cdots+f(a_n)}{n} \right).
}

It is easy to verify that for
$I=\R_+$ and $f=\pi_p$, where $\pi_p(x):=x^p$ if $p\ne 0$ and $\pi_0(x):=\ln x$, then the mean $\QA{f}$ coincides with the $p$-th power mean (from now on denoted by $\P_p$); this fact had been already noticed by Knopp \cite{Kno28} before quasi-arithmetic means were formally introduced.

In the course of dealing with the Gaussian iteration process we will use the notation of Arrow-Pratt index \cite{Arr65,Pra64}, which was also investigated by Mikusi\'nski \cite{Mik48}.
Whenever $f \colon I \to \R$ is twice differentiable with nowhere vanishing first derivative we can define the operator $\Mik{f}:=f''/f'$. It can be proved that comparability of quasi-arithmetic means is equivalent to pointwise comparability of respective Arrow-Pratt indexes (see \cite{Mik48} for details).

Following the idea from \cite{Pas16a} we will assume that all consider function are smooth enough to apply operator $\Mik{}$. Moreover, for technical reasons, we assume that second derivative is of {\it almost bounded variation} (finite variation restricted to every compact interval; cf. \cite[p.~135]{KucChoGer90}). 
Using this definition we introduce the class 
\Eq{*}{
\Sm(I):=\{f \in \mathcal{C}^2(I) \colon f' \ne 0 \text{ and } f'' \text{ is of almost bounded variation }\}.
}
Obviously, as $f' \ne 0$, each element belonging to $\Sm(I)$ is a continuous and strictly monotone function, and therefore it generates the quasi-arithmetic mean. The assumption that $f''$ is of almost bounded variation is technical, however important from the point of view of the present paper (this is also the setting which was extensively used in the previous paper \cite{Pas16a}). 

Following the idea from \cite{Pas16a} we are going to deal with the Gaussian iteration of quasi-arithmetic means. Define, for the vector $\bff=(f_j)_{j=1}^n$ of continuous, strictly monotone functions on $I$, the mapping $\QAself{\bff} \colon I^n \to I^n$ by 
\Eq{*}{
\QAself{\bff}(a):=\big(\QA{f_1}(a),\dots,\QA{f_n}(a)\big).
}
In fact $\QAself{\bff}$ is the quasi-arithmetic counterpart of the function $\Mb$, which appears in the definition of Gaussian product. Then it is known that there exists a unique continuous and strict mean $\QA{\otimes} \colon I^n \to I$ such that $\QA{\otimes} \circ \QAself{\bff}=\QA{\otimes}$. It has also further implications but let us introduce some necessary notations first. For a vector $a$ of real numbers we denote its arithmetic mean, variance, and spread briefly by $\oa$, $\Var(a)$, and $\delta(a):=\max(a)-\min(a)$, respectively.

It is known that for every vector $a \in I^n$, the sequence $(\Var(\QAself{\bff}^k(a)))_{k \in \N}$ tends to zero. Moreover, due to \cite{Pas16a}, if $\bff \in \Sm(I)^n$  then this convergence is double exponentially with fractional base. We will prove that, in a non-degenerated case, this sequence tends to zero quadratically and, moreover, we will calculate the limit
\Eq{*}{
\lim_{k \to \infty} \frac{\Var(\QAself{\bff}^{k+1}(a))}{\Var(\QAself{\bff}^k(a))^2}.
}

%\vskip25mm\hline\vskip25mm

\subsection{Approximate value of quasi-arithmetic means} We are now heading towards calculation of quasi-arithmetic mean in the spirit of Taylor. In fact the crucial identity was already established in the previous paper. Let us recall this result (Riemann–Stieltjes integral is used in its wording).

\begin{lem}[\cite{Pas16a}, Lemma 4.1]
\label{lem:QAexpansion}
For every $f \in \Sm(I)$ and $a \in I^n$, $n \in \N$,
\Eq{*}{
\QA{f}(a)&=\oa + \tfrac12 \Var(a) \cdot \Mik{f}(\oa) +R_f(a)+S_f(a),
}
where
\Eq{*}{
R_f(a)&:= \frac{1}{2n \cdot f'(\oa)} \cdot \sum_{i=1}^n \int_{\oa}^{a_i} (a_i -t)^2 df''(t), \\
S_f(a)&:= \int_{\oa}^{\QA{f}(a)} \frac{\big(f(u)-f(\QA{f}(a))\big)f''(u)}{f'(u)^2}du.
}
\end{lem}

It was also proved \cite[Lemma 4.2]{Pas16a} that 
\Eq{Lem4.2}{
\abs{R_f(a)} \le \frac1{6k}  \cdot \exp(\norm{\Mik{f}}_*) \cdot \sum_{i=1}^k \abs{a_i-\oa}^3, \quad
\abs{S_f(a)} \le  (\QA{f}(a) - \oa)^2 \cdot \exp(\norm{\Mik{f}}_*),}
where the $*$-norm is defined as $\norm{g}_*:=\sup\limits_{a,\,b \in \:\dom(g)} \abs{\int_a^b g(t)dt}$. 

What was not noticed is that if the second derivative of $f$ is locally Lipschitz then the error terms can be majorized much more efficient. We are going to prove this in a while. First, define $\SmLip(I):= \{f \in \Sm(I) \colon f'' \text{ is locally Lipschitz}\}$; $\Sm_K(I):=\{f \in\Sm(I) \colon \norm{\Mik{f}}_\infty \le K\}$ for $K>0$ and $\SmLip_K(I):=\SmLip(I)\cap \Sm_K(I)$.

For the purpose of this estimation let us make purely technical assumption $K=1$, which will be omitted soon.

\begin{lem}
\label{lem:Rest_estim}
For every $f \in \SmLip_1(I)$ and $a \in I^n$, $n \in \N$,
\Eq{*}{
|R_f(a)| \le \frac{\Lip(f'')}{2\abs{f'(\oa)}} \cdot \delta(a) \Var(a) \qquad\text{ and }\qquad |S_f(a)| \le \frac{\alpha^2}{4} \exp(\norm{\Mik{f}}_*) \delta(a)^4,
}
where $\alpha:=\tfrac{3+7e}3$.
\end{lem}

\begin{proof}
By mean-value theorem there exist $\xi_1,\dots,\xi_n,\eta \in (\min a,\,\max a)$ such that
\Eq{*}{
R_f(a)&= \frac{1}{2n \cdot f'(\oa)} \cdot \sum_{i=1}^n \int_{\oa}^{a_i} (a_i -t)^2 df''(t) \\
&= \frac{1}{2n \cdot f'(\oa)} \cdot \sum_{i=1}^n  \Big( -(a_i-\oa)^2 f''(\oa) - 2 \int_{\oa}^{a_i} (a_i-t) f''(t) dt \Big)\\
&= \frac{1}{2n \cdot f'(\oa)} \cdot \sum_{i=1}^n (a_i-\oa)^2 \big(f''(\xi_i)-f''(\oa) \big) \\
&= \frac{1}{2n} \cdot \sum_{i=1}^n (a_i-\oa)^2 \frac{f''(\eta)-f''(\oa)}{f'(\oa)}
= \frac{\Var(a)}{2}  \cdot \frac{f''(\eta)-f''(\oa)}{f'(\oa)}.
}
Therefore
\Eq{*}{
|R_f(a)| = \frac{\abs{\eta-\oa} \Var(a)}{2\abs{f'(\oa)}}  \cdot \abs{\frac{f''(\eta)-f''(\oa)}{\eta-\oa}} \le \frac{\Lip(f'')}{2\abs{f'(\oa)}} \cdot \delta(a) \Var(a).
}

We will now prove the second inequality. By \eq{Lem4.2}, we have 
\Eq{*}{
|S_f(a)| \le  (\QA{f}(a) - \oa)^2 \cdot \exp(\norm{\Mik{f}}_*).
}
Furthermore, by \cite[Lemma 4.3]{Pas16a}, we get $\abs{\QA{f}(a)-\oa} \le \frac{\alpha}2 \delta(a)^2$. Thus
\Eq{*}{
|S_f(a)| &\le \frac{\alpha^2}{4} \exp(\norm{\Mik{f}}_*) \delta(a)^4,
}
what was to be proved.
\end{proof}

\section{Main result}
Binding two results above we can establish the main theorem of the present note. 
In order to make the notation more compact the brief sum-type notation of mean will be used (that is we will write $\Mm_{k=1}^n(t_k)$ instead of $\M(t_1,\dots,t_n)$\,). 
Additionally, for the same reason, we will use the $\pm$ notation of the remainder (with the natural interpretation).

\begin{thm}
Let $I$ be an interval, $K>0$, $n\in \N$, $(f_j)_{j=1}^n \in \SmLip_K(I)^n$, and $a$ be a vector having entries in $I$. Then
\Eq{E:main}{
\Var(\QAself{\bff}(a))=\frac14 \Var(a)^2 \Var(\Mik{\bff}(\oa)) \pm 4CK^5\delta(a)^5 \pm (3C^2+C_2^2) K^6 \delta(a)^6,
}
where $\Mik{\bff}\colon I \to \R^n$ is defined by $\Mik\bff(x):= (\Mik{f_1}(x),\dots,\Mik{f_n}(x))$, $\alpha:=\tfrac{3+7e}3$, and
\Eq{*}{
C:=\Ar_{k=1}^n \left(\frac{\Lip(f_k'')}{2K^2\abs{f_k'(\oa)}}+\frac{\alpha^2e}4\right),\quad 
C_2:=\op{\P_2}_{k=1}^n \left(\frac{\Lip (f_k'')}{2K^2\abs{f_k'(\oa)}}+\frac{\alpha^2e}4 \right).
}
\end{thm}

Recall that $\A$ and $\P_2$ stand for arithmetic and quadratic mean, respectively.

\begin{proof}
Applying the machinery described in \cite[section 4.1]{Pas16a} we can use the mapping 
\Eq{*}{
\SmLip_K(I) \ni f(x) \mapsto f(x/K) \in \SmLip_1(K \cdot I)
}
to each function $f_k$. Therefore we will assume, without loss of generality, that $K=1$. In fact to make such assumption possible, we need to verify that the statement in the theorem in both setups are equivalent. Precise calculations are not very simple, but rather straightforward. 

In the case when $\delta(a)\ge1$ we have $C \ge \alpha^2e/4 >1$, thus the admissible error on the right hand side is at least $3 \delta(a)^6$.
Meanwhile
\Eq{*}{
\abs{\Var(\QAself{\bff}(a))-\frac14 \Var(a)^2 \Var(\Mik{\bff}(\oa))} &\le \Var(\QAself{\bff}(a))+\frac14 \Var(a)^2 \Var(\Mik{\bff}(\oa)) \\
&\le \delta(a)^2+ \frac{\delta(a)^4}4
= \frac 54 \delta(a)^4
\le \frac 54 \delta(a)^6.
}

From now on we will assume that $\delta(a)<1$. By Lemmas \ref{lem:QAexpansion} and \ref{lem:Rest_estim}, 

\Eq{*}{
\abs{\Ar_{k=1}^n \QA{f_k}(a)-\oa-\Ar_{k=1}^n \tfrac12 \Var(a) \cdot \Mik{f_k}(\oa) }&=  \abs{\Ar_{k=1}^n R_{f_k}(a)+S_{f_k}(a)} \\
\le \Ar_{k=1}^n \frac{\Lip(f_k'')}{2\abs{f_k'(\oa)}} &\cdot \delta(a) \Var(a) + \frac{\alpha^2}{4} \Ar_{k=1}^n \exp(\norm{\Mik{f_k}}_*) \delta(a)^4.
}

We know that $\Var(a) \le \delta(a)^2$, thus we obtain
\Eq{*}{
\Ar_{k=1}^n \QA{f_k}(a)=\oa+\Ar_{k=1}^n \tfrac12 \Var(a) \cdot \Mik{f_k}(\oa)  \pm \Ar_{k=1}^n \frac{\Lip(f_k'')}{2\abs{f_k'(\oa)}} \cdot \delta(a)^3 \pm \frac{\alpha^2}{4} \Ar_{k=1}^n \exp(\norm{\Mik{f_k}}_*) \delta(a)^4.
}
As $\delta(a)<1$ we get $\delta(a)^4 \le \delta(a)^3$ and $\exp(\norm{\Mik{f_k}}_*) \le e$. Therefore
\Eq{*}{
\Ar_{k=1}^n \QA{f_k}(a)=\oa+\Ar_{k=1}^n \tfrac12 \Var(a) \cdot \Mik{f_k}(\oa)  \pm \Big( \Ar_{k=1}^n \frac{\Lip(f_k'')}{2\abs{f_k'(\oa)}} + \frac{\alpha^2e}4 \Big) \cdot \delta(a)^3.
}
We can express it briefly as 
\Eq{*}{
\Ar_{k=1}^n \QA{f_k}(a)=\oa+\Ar_{k=1}^n \tfrac12 \Var(a) \cdot \Mik{f_k}(\oa)  \pm C \delta(a)^3.
}
Thus, using Lemmas \ref{lem:QAexpansion} and \ref{lem:Rest_estim} again, we have 
\Eq{*}{
\QA{f_j}(a)-\Ar_{k=1}^n \QA{f_k}(a)= \tfrac12 \Var(a) \cdot \Big(\Mik{f_j}(\oa)-\Ar_{k=1}^n\Mik{f_k}(\oa)\Big)  \pm \Big(C+\frac{\Lip(f_j'')}{2\abs{f_j'(\oa)}}+\frac{\alpha^2e}4\Big) \delta(a)^3.
}

Therefore
\Eq{*}{
\Big(\QA{f_j}(a)-\Ar_{k=1}^n \QA{f_k}(a)\Big)^2&= \tfrac14 \Var(a)^2 \cdot \Big(\Mik{f_j}(\oa)-\Ar_{k=1}^n\Mik{f_k}(\oa)\Big)^2 \\
&\qquad \pm \Big(C+\frac{\Lip(f_j'')}{2\abs{f_j'(\oa)}}+\frac{\alpha^2e}4\Big) \delta(a)^3 \Var(a) \cdot \abs{\Mik{f_j}(\oa)-\Ar_{k=1}^n\Mik{f_k}(\oa)}\\
&\qquad \pm \Big(C+\frac{\Lip(f_j'')}{2\abs{f_j'(\oa)}}+\frac{\alpha^2e}4\Big)^2 \delta(a)^6.
}
But, by $\abs{\Mik{f_k}}\le 1$ we get $\abs{\Mik{f_j}(\oa)-\Ar_{k=1}^n\Mik{f_k}(\oa)} \le 2$, moreover $\Var(a) \le \delta(a)^2$. Whence
\Eq{*}{
\Big(\QA{f_j}(a)-\Ar_{k=1}^n \QA{f_k}(a)\Big)^2&= \tfrac14 \Var(a)^2 \cdot \Big(\Mik{f_j}(\oa)-\Ar_{k=1}^n\Mik{f_k}(\oa)\Big)^2 \\
&\pm 2 \cdot \Big(C+\frac{\Lip(f_j'')}{2\abs{f_j'(\oa)}}+\frac{\alpha^2e}4\Big) \delta(a)^5
\pm \Big(C+\frac{\Lip(f_j'')}{2\abs{f_j'(\oa)}}+\frac{\alpha^2e}4\Big)^2 \delta(a)^6.
}

We now apply the operator $\Ar_{j=1}^n$ side-by-side to the equality above to obtain
\Eq{X1}{
\Var(\QAself{\bff}(a))&=\frac14 \Var(a)^2 \Var(\Mik{\bff}(\oa)) 
\pm \Ar_{j=1}^n \Big(2C+\frac{\Lip(f_j'')}{\abs{f_j'(\oa)}}+\frac{\alpha^2e}2\Big) \delta(a)^5 \\
&\qquad \pm \Ar_{j=1}^n \Big(C+\frac{\Lip(f_j'')}{2\abs{f_j'(\oa)}}+\frac{\alpha^2e}4\Big)^2 \delta(a)^6.
}
But 
\Eq{X2}{
\Ar_{j=1}^n \Big(2C+\frac{\Lip(f_j'')}{\abs{f_j'(\oa)}}+\frac{\alpha^2e}2\Big)
=2C+\Ar_{j=1}^n \Big(\frac{\Lip(f_j'')}{\abs{f_j'(\oa)}}+\frac{\alpha^2e}2\Big)=4C.
}
Additionally
\Eq{X3}{
\Ar_{j=1}^n \Big(C+\frac{\Lip(f_j'')}{2\abs{f_j'(\oa)}}+\frac{\alpha^2e}4\Big)^2
&=C^2+2C \cdot \Ar_{j=1}^n \Big(\frac{\Lip(f_j'')}{2\abs{f_j'(\oa)}}+\frac{\alpha^2e}4\Big)+\Ar_{j=1}^n \Big(\frac{\Lip(f_j'')}{2\abs{f_j'(\oa)}}+\frac{\alpha^2e}4\Big)^2\\
&=C^2+2C^2+C_2^2=3C^2+C_2^2.
}
Binding \eq{X1}, \eq{X2}, and \eq{X3} we obtain the final statement.
\end{proof}

\begin{xrem}
As the values of $f_j$-s outside the interval $[\min a,\max a]$ do not affect to the left hand side of the inequality \eq{E:main}, we can simply assume that $I=[\min a,\max a]$ i.e. take a Lipschitz constant on the restricted domain only.
\end{xrem}

\begin{cor}
 Let $\bff=(f_1,\dots,f_n) \in \SmLip(I)^n$ and $a \in I^n$. Consider the mapping $\QAself{\bff}:=(\QA{f_1},\dots,\QA{f_n}) \colon I^n \to I^n$.
 Then either the iteration process $\QAself{\bff}$ is degenerated or 
 \Eq{*}{
 \lim_{k \to \infty} \frac{\Var \QAself{\bff}^{k+1}(a)}{\big(\Var \QAself{\bff}^{k}(a)\big)^2}=\frac{\Var\big(\Mik{\bff}(\QA{\otimes}(a))\big)}4.
 }
\end{cor}
\begin{proof}
Assume that the iteration process is not degenerated. Applying the machinery described in \cite[section 4.1]{Pas16a} we can assume that $\bff \in \SmLip_1(I)^n$.
We know that 
\Eq{E:Varineq}{
\Var(a) \in (\delta(a)^2/2n, \delta(a)^2).
}
Thus, if we divide \eq{E:main} side-by-side by $\Var(a)^2$ we get
\Eq{*}{
\frac{\Var(\QAself{\bff}(a))}{\Var(a)^2}&=\frac14 \Var(\Mik{\bff}(\oa)) \pm 16n^2C \delta(a) \pm 4n^2(3C^2+C_2^2)  \delta(a)^2.
} 
If we now put $a \leftarrow \QAself{\bff}^k(a)$, we obtain
\Eq{*}{
\frac{\Var(\QAself{\bff}^{k+1}(a))}{\Var(\QAself{\bff}^k(a))^2}&=\frac14 \Var\Big(\Mik{\bff}\Big(\overline{\QAself{\bff}^k(a)}\Big)\Big) \pm 16n^2C \delta(\QAself{\bff}^k(a)) \pm 4n^2(3C^2+C_2^2)  \delta(\QAself{\bff}^k(a))^2.
}
But we know that $\delta(\QAself{\bff}^k(a)) \to 0$ and $\QAself{\bff}^k(a) \to m$ for all $k$. Therefore
\Eq{*}{
\lim_{k \to \infty } \frac{\Var(\QAself{\bff}^{k+1}(a))}{\Var(\QAself{\bff}^k(a))^2}&=\frac{\Var\big(\Mik{\bff}(m)\big)}4,
}
what concludes the proof.
\end{proof}

By the property \eq{E:Varineq} we also obtain
\begin{cor}
 Let $\bff=(f_1,\dots,f_n) \in \SmLip(I)^n$ and $a \in I^n$. Consider a mapping $\QAself{\bff}:=(\QA{f_1},\dots,\QA{f_n}) \colon I^n \to I^n$.
 Then either the iteration process $\QAself{\bff}$ is degenerated or $(\delta(\QAself{\bff}^k))_{k=1}^{\infty}$ tends to zero quadratically.
\end{cor}

\def\cprime{$'$} \def\R{\mathbb R} \def\Z{\mathbb Z} \def\Q{\mathbb Q}
  \def\C{\mathbb C}

\end{document}